\documentclass{article}
\usepackage{graphicx} % Required for inserting images
\usepackage{amscd}
\usepackage{authblk}
\usepackage{amsmath}
\usepackage{amsfonts}
\usepackage{amssymb}
\usepackage{enumerate}
\usepackage{graphicx}
\usepackage{latexsym}
\usepackage{color}
\usepackage{soul}
\usepackage[a4paper,
            bindingoffset=0.2in,
            left=1.3in,
            right=1.3in,
            top=1.5in,
            bottom=1.2in,
            footskip=.25in]{geometry}
\numberwithin{equation}{section}
\newtheorem{theorem}{Theorem}[section]

\newtheorem{corollary}{Corollary}[section]

\newtheorem{definition}{Definition}[section]
\newtheorem{example}{Example}[section]

\newtheorem{lemma}{Lemma}[section]

\newenvironment{proof}{{\bf Proof.}}{{\hspace*{\fill}$\Box$}}

\title{New averaged type algorithms for solving split common fixed-point problem for demicontractive mappings}
\author[1,2]{Vasile Berinde}
\author[3]{Khairul Saleh}

\affil[1]{\footnotesize Department of Mathematics and Computer Science Technical University of Cluj-Napoca, North University Centre at Baia Mare, Victoriei 76, 430122 Baia Mare, Romania}

\affil[2]{\footnotesize Academy of Romanian Scientists, 3 Ilfov, 050044, Bucharest, Romania}

\affil[3]{\footnotesize Department of Mathematics, King Fahd University of Petroleum
and Minerals, Dhahran 31261, Saudi Arabia}

\date{}

\begin{document}
\maketitle
\begin{abstract}
In this paper we propose new averaged iterative algorithms designed for solving a split common fixed-point problem in the class of demicontractive mappings. The algorithms are obtained by inserting an  averaged term into the algorithms used in [Li, R. and He, Z., A new iterative algorithm for split solution problems of quasi-nonexpansive mappings {\it J. Inequal. Appl.} {\bf 131} (2015), 1--12.] for solving the same problem but in the class of quasi-nonexpansive mappings, which is a subclass of demicontractive mappings. Basically, our investigation is based on the embedding of demicontractive operators in the class of quasi-nonexpansive operators by means of averaged mappings. For the considered algorithms we prove weak and strong convergence theorems  in the setting of a real Hilbert space and also provide examples to show that our results are effective generalizations of existing results in literature.

%ove convergence theorems for approximating the split common fixed points of two demicontractive mappings in Hilbert spaces derived from the corresponding convergence theorems in the class of quasi-nonexpansive mappings. Our investigation is based on the embedding of demicontractive operators in the class of quasi-nonexpansive operators by means of averaged mappings.
\end{abstract}

$2010$ \textit{Mathematics Subject Classification:} 47H10; 47H09.
\section{Introduction}
Let ${\cal H}$ be a real Hilbert space with norm $\|\cdot\|$ and inner product $\langle \cdot,\cdot\rangle.$ Let $D\subset {\cal H}$ be a closed convex set, and consider the self-mapping $G:D\to D.$  Throughout this paper the set of all fixed point of G in $D$ is denoted by
\[
Fix(G)=\{u\in D~:~Gu=u\}.
\]
The mapping $G$ is said to be
\begin{enumerate}
  \item [(a)] \emph{nonexpansive} if
  \begin{equation}\label{ne}
    \|Gu-Gv\|\leq\|u-v\|,\quad \text{for all}~u,v\in D;
  \end{equation}
  \item [(b)] \emph{quasi-nonexpansive} if $Fix(G)\neq \emptyset$ and
  \begin{equation}\label{qne}
    \|Gu-v\|\leq \|u-v\|,\quad \text{for all}~u\in D~\text{and}~v\in Fix(G);
  \end{equation}
  \item [(c)] \emph{$\beta$-demicontractive} if $Fix(G)\neq \emptyset$ and there exists a positive number $\beta<1$ such that
  \begin{equation}\label{dc}
    \|Gu-v\|^2\leq \|u-v\|^2+\beta\|u-Gu\|^2,
  \end{equation}
  for all $u\in D$ and $v\in Fix(G)$.
\end{enumerate}
By the previous definitions, it is obvious that any nonexpansive mapping $G$ with $Fix(G)\neq \emptyset$ is demicontractive and that any quasi-nonexpansive mapping is demicontractive, too, but the reverses are no more true, as illustrated by the next example.

\begin{example} [\cite{Berinde2024}]\label{ex1}
Let ${\cal H}$ be the real line with the usual norm and $D=[0,1].$ Define $F$ on $D$ as
\begin{equation}\label{Ex1}
F(u)=\begin{cases}
  7/8, & \mbox{if } 0\leq u<1 \\
  1/4, & \text{if}~u=1.
\end{cases}
 \end{equation}
Then $F$ is demicontractive but $F$ is neither quasi-nonexpasive nor nonexpansive.
\end{example}

Let ${\cal H}_1$ and ${\cal H}_2$ be two real Hilbert spaces and $F: {\cal H}_1\to {\cal H}_1,~ G : {\cal H}_2\to {\cal H}_2$ be two nonlinear
mappings with $Fix(F)\neq\emptyset$ and $Fix(G)\neq\emptyset.$ The split common fixed point problem for $F$ and $G$ is to
\begin{equation}\label{sfp}
\text{find an element}~ u \in Fix(F)~ \text{such that}~ Au\in Fix(G).
\end{equation}
where $A : {\cal H}_1 \to {\cal H}_2$ is a bounded linear operator. We denote the set of all solutions of Problem (\ref{sfp}) by
\[
\Gamma=\{u\in Fix(F): Au\in Fix(G)\}.
\]
This problem was first introduced by Censor and Segal \cite{Cencor} in 2009. It was mentioned in \cite{Cencor} that the split common fixed point problem is a generalization of the split and the convex feasibility problems that are very useful in image recovery, convex optimization and other areas of applied mathematics. Some convergence theorems for the split common fixed point problem have been analyzed, see \cite{LiHe, Moudafi} and the references therein. For instance, Moudafi \cite{Moudafi} proved weak convergence of the following iteration method for two quasi-nonexpasive mappings $F$ and $G,$ assuming that $F-I$ and $G-I$ are demiclosed at zero.

Let $u_1\in{\cal H}_1.$ For all $p\in\mathbb{N},$ define
\begin{equation}\label{eq-1}
\begin{cases}
y_p=u_p+\gamma\mu A^*(G-I)Au_p,\\
u_{p+1}=(1-\alpha_p)y_p+\alpha_p F(y_p),
\end{cases}
\end{equation}
where $\mu\in(0,1),~\{\alpha_p\}\subset(\delta,1-\delta)$ for a small enough $\delta>0$ and $\gamma\in(0,\frac{1}{\lambda\mu}),$ with $\lambda$ being the spectral radius of $A^*A.$ Here $A^*$ denotes the adjoint of $A.$

It was established in Moudafi \cite{Moudafi} that the sequence $\{u_p\}$ given by \eqref{eq-1} converges weakly to a split common fixed point of $F$ and $G$.

A few years later, Li and He \cite{LiHe} introduced a new iteration scheme with certain conditions that was shown to converge strongly to a split common fixed point of two quasi-nonexpansive mappings.

In this paper, our aim is to solve the split common fixed point problem in the setting  Hilbert spaces for the case of the larger class of demicontractive mappings, thus extending the main results in Li and He \cite{LiHe}. Our results are obtained by considering new averaged iterative algorithms and prove weak and strong convergence theorems.

%we are interested in the split fixed point problem in the setting of demicontractive mappings in Hilbert spaces. Inspired by \cite{Berinde2023,Berinde2024} we prove weak and strong convergence theorems for the split common fixed point problem in Hilbert space for two self-mappings that are demicontractive.

\section{Preliminaries}

We recall some important lemmas used in the proofs of our main results. The following two lemmas are due to Berinde \cite{Berinde2023}.

\begin{lemma} {\rm\cite{Berinde2023}} \label{L1}
Let ${\cal H}$ be a real Hilbert space and $D\subset {\cal H}$ a closed and convex set. If $G:D\to D$ is $\beta$-demicontractive, then the Krasnoselskij perturbation $G_{\nu}=(1-\nu)I+\nu G$ of $G$ is $(1+\beta/\nu-1/\nu)$-demicontractive.
\end{lemma}
%The following lemma is essential in the proofs of our main convergence theorems.

\begin{lemma} {\rm\cite{Berinde2023}} \label{L2}
Let ${\cal H}$ be a real Hilbert space and $D\subset {\cal H}$ a closed and convex set. If $G:D\to D$ is $\beta$-demicontractive, then for any $\nu\in(0,1-\beta)$ $$G_{\nu}=(1-\nu)I+\nu G $$ is quasi-nonexpansive.
\end{lemma}

\begin{lemma} \label{L2a}
Let ${\cal H}$ be a real Hilbert space, $D\subset {\cal H}$ a closed and convex set and  $F:D\to D$ a mapping. Then, for any $\nu\in(0,1),$ we have $Fix(F_{\nu})=Fix(F).$
\end{lemma}

The next lemma is known as Opial's lemma.
\begin{lemma} [\cite{Opial}] \label{Opial}
Let ${\cal H}$ be a Hilbert space and $\{u_p\}\subset{\cal H}$ a sequence such that there exists a nonempty set $C\subset {\cal H}$ satisfying the following.
\begin{enumerate}
  \item [(a)] $\lim_{p\to\infty}\|u_p-v\|$ exists $\forall v\in C.$
  \item [(b)] Any weak-cluster point of the sequence $\{u_p\}$ belongs to $C.$
\end{enumerate}
Then there exists $\bar{u}\in C$ such that $\{u_p\}$ weakly converges to $\bar{u}.$
\end{lemma}
The next lemma is due to Li and He \cite{LiHe}.
\begin{lemma}{\rm\cite{LiHe}}
Let $F_1,\cdots,F_n:{\cal H}_1\to {\cal H}_1$ be quasi-nonexpansive mappings and set $T=\sum_{i=1}^{n}b_iF_{a_i},$ where $b_i\in(0,1)$ with $\sum_{i=1}^{n}b_i=1,$ and $F_{a_i}=(1-a_i)I+a_iF_i$ with $a_i\in(0,1), i=1,2,\cdots,n.$ Then $T$ is quasi-nonexpansive and $$Fix(T)=\bigcap_{i=1}^{n}Fix(F_i)=\bigcap_{i=1}^{n}Fix(F_{a_i}).$$
\label{LF}
\end{lemma}
%\begin{definition}
%A sequence $\{u_n\}$ is said to be fej\'er-monotone if
%\[
%\|u_{p+1}-u\|\leq \|u_p-u\|
%\]
%for every $p\in\mathbb{N}$ and $u\in\Gamma.$
%\end{definition}
\begin{lemma} {\rm\cite{Moudafi}}
Let $S$ be a quasi-nonexpansive mapping and set $S_a:=(1-a)I+aS$ for $a\in(0,1].$ The following properties hold for all $u\in{\cal H}$ and $v\in Fix(S).$
\begin{enumerate}
  \item [{\rm(1)}] $\langle u-Su,u-v\rangle\geq \frac{1}{2}\|u-Su\|^2$ and $\langle u-Su,v-Su\rangle\leq \frac{1}{2}\|u-Su\|$;
  \item [{\rm(2)}] $\|S_au-v\|^2\leq \|u-v\|^2-a(1-a)\|u-Su\|$;
  \item [{\rm(3)}] $\langle u-S_au,u-v\rangle\geq \frac{a}{2}\|u-Su\|^2.$
\end{enumerate}
\label{L3}
\end{lemma}

%\begin{lemma} {\rm\cite{Moudafi}}
%Let ${\cal H}_1$ and ${\cal H}_2$ be Hilbert spaces. Let $F:{\cal H}_1\to{\cal H}_1$ and $G:{\cal H}_2\to{\cal H}_2$ be two quasi-nonexpansive operators with $Fix(F)=P$ and $Fix(G)=Q.$ Consider a bounded operator $A:{\cal H}_1\to{\cal H}_2.$ Any sequence $\{u_p\}$ generated by {\rm(\ref{eq-1})} is Fej\'er-monotone with respect to $\Gamma$,
%that is, for every $v\in\Gamma,$
%\[
%\|u_{p+1}-v\|\leq \|u_p-v\|,\quad \forall p\in\mathbb{N},
%\]
%provided that $\gamma\in (0,\frac{1}{\lambda \mu}]$ and $\alpha_p\in (0,1]$ for all $p\in\mathbb{N}.$
%\label{LFejer}
%\end{lemma}
%\hl {Define Fejer sequences here and required properties}

\section{Split Common Fixed Points for Demicontractive mappings}

In this section we state and prove our main results: a weak convergence theorem (Theorem \ref{MT1}) and a strong convergence theorem (Theorem \ref{MT2}) for the averaged type algorithm \eqref{Alg1} used to approximate a split common fixed point of two demicontractive mappings.

We also obtain a common fixed point result for two demicontrative mappings as a special case of the main convergence theorem (Theorem \ref{MT2}).
\begin{definition} {\rm\cite{Moudafi}}
A sequence $\{u_p\}$ is said to be Fej\'er-monotone with respect to $\Gamma$ if for every $v\in\Gamma,$
\[
\|u_{p+1}-v\|\leq \|u_p-v\|,\quad \forall p\in\mathbb{N}.
\]
\end{definition}

\begin{theorem}
Let ${\cal H}_1$ and ${\cal H}_2$ be two real Hilbert spaces. Given a bounded linear
operator $A : {\cal H}_1 \to {\cal H}_2,$ let $F: {\cal H}_1\to {\cal H}_1,~ G : {\cal H}_2\to {\cal H}_2$ be two $\beta$-demicontractive operators. 
%\hl {with $Fix(F)\neq\emptyset$ and $Fix(G)\neq\emptyset.$} 
Assume that $F-I$ and $G-I$ are demiclosed at zero. Let $\{u_p\}$ be a sequence generated by
\begin{equation}\label{Alg1}
\begin{cases}
  u_1\in{\cal H}_1, &\\
  y_p=u_p+\gamma\mu A^*(b(G-I))Au_p,\quad b\in(0,1-\beta), &\\
  u_{p+1}=(1-\alpha_p)y_p+\alpha_p F(y_p),~~\forall p\in\mathbb{N},
\end{cases}
\end{equation}
where $\mu\in(0,1),~\{\alpha_p\}\subset(\delta,1-\delta)$ for a small enough $\delta>0$
%\hl {$G_{\nu}=(1-\nu)u+\nu Gu $ to be deleted (I put in in the equation)}
and $\gamma\in(0,\frac{1}{\lambda\mu})$ with $\lambda$ being the spectral radius of $A^*A$ and $A^*$ being the adjoint of $A.$ Then $\{u_p\}$ converges weakly to a split common fixed point $u^*\in\Gamma.$
\label{MT1}
\end{theorem}
\begin{proof}
Since $F$ and $G$ are $\beta$-demicontractive mappings, in view of Lemma \ref{L2} the averaged mappings
\begin{align*}
F_{a}=(1-a)I+a F \quad \text{and} \quad G_{b}=(1-b)I+b G
\end{align*}
are quasi-nonexpansive for $\{a,b\}\subset(0,1-\beta).$ Clearly, $F_{a}-I$ and $G_{b}-I$ are demiclosed at zero.

For any sequence $\{t_p\}\subset(\delta,1-\delta)$ with a small enough $\delta>0,$ consider the iteration $\{u_p\}$ generated by
\begin{equation*}
\begin{cases}
  u_1\in{\cal H}_1, &\\
  y_p=u_p+\gamma\mu A^*(G_b-I)Au_p, &\\
  u_{p+1}=(1-t_p)y_p+t_p F_{a}y_p,~~\forall p\in\mathbb{N}.
\end{cases}
\end{equation*}
Take $v\in\Gamma.$ By Property (2) in Lemma \ref{L3}, we obtain
\begin{equation}\label{EqN}
 \|u_{p+1}-v\|^2\leq \|y_k-v\|^2-t_p(1-t_p)\|G_b(y_p)-y_p\|^2.
\end{equation}
We also have
\begin{align*}
  \|y_p-v\|^2 &= \|u_p+\gamma\mu A^*(G_b-I)(Au_p)-v\|^2 \\
   & =\|u_p-v\|^2+\gamma^2\mu^2 \|(G_b-I)(Au_p)\|^2+2\gamma\mu\left\langle u_p-v,A^*(G_b-I)(Au_p)\right\rangle \\
   & =\|u_p-v\|^2+\gamma^2\mu^2 \left\langle(G_b-I)(Au_p),AA^*(G_b-I)(Au_p)\right\rangle\\
   &\quad +2\gamma\mu\left\langle u_p-v,A^*(G_b-I)(Au_p)\right\rangle.
\end{align*}
Let $\Delta:=2\gamma\mu\langle u_p-v, A^*(G_p-I)(Au_p)\rangle.$ Propery (1) of Lemma \ref{L3} yields
\begin{align*}
  \Delta &= 2\gamma\mu\langle A(u_p-v),(G_p-I)(Au_p)\rangle \\
  & =2\gamma\mu \langle A(u_p-v)+(G_p-I)(Au_p)-(G_p-I)(Au_p),(G_p-I)(Au_p)\rangle\\
  & =2\gamma\mu \left(\langle T(Au_p)-A(v),(G_p-I)(Au_p)\rangle-\|(G_p-I)(Au_p)\|^2\right)\\
  & \leq 2\gamma\mu\left(\frac{1}{2}\|(G_p-I)(Au_p)\|^2-\|(G_p-I)(Au_p)\|^2\right)\\
  & \leq -\gamma\mu\|(G_p-I)(Au_p)\|^2.
\end{align*}

From the definition of $\lambda,$ we have
\begin{align*}
  \gamma^2\mu^2 \langle(G_b-I)(Au_p),AA^*(G_b-I)(Au_p)\rangle & \leq \lambda\gamma^2\mu^2\langle(G_b-I)(Au_p),(G_b-I)(Au_p)\rangle \\
   & = \lambda \gamma^2\mu^2  \|(G_b-I)(Au_p)\|^2.
\end{align*}
From \ref{EqN}, we obtain
\begin{equation*}
\|u_{p+1}-v\|^2\leq \|u_p-v\|^2-\gamma\mu(1-\lambda\gamma\mu)\|(G_b-I)(Au_p)\|^2-t_p(1-t_p)\|F_a(y_p)-y_p\|^2.
\end{equation*}
Thus, $\|u_{p+1}-v\|\leq \|u_p-v\|,$ that is, the sequence $\{u_p\}$ is Fej\'er-monotone.

From the assumption that $\{\alpha_p\}\subset(\delta,1-\delta)$ and $\gamma\in(0,\frac{1}{\lambda\mu}),$ we have
%\hl { we should avoid citing this. Instead, the inequality should be deduced completely } %we have \cite{Moudafi}
\begin{equation}\label{Eqeq}
\|u_{p+1}-v\|^2\leq \|u_p-v\|^2-\gamma\mu(1-\lambda\gamma\mu)\|(G_b-I)(Au_p)\|^2-\delta^2\|F_a(y_p)-y_p\|^2
\end{equation}
for any $v\in\Gamma.$ This implies that the sequence $\{\|u_p-v\|\}$ is monotone decreasing, and thus converges to a positive number. Hence,
\begin{equation}\label{Eqeq2}
  \lim_{p\to\infty}\|(G_b-I)(Au_p)\|=0\quad \text{and}\quad \lim_{p\to\infty}\|F_a(y_p)-y_p\|=0.
\end{equation}
Fej\'er monotonicity of $\{u_p\}$ implies that the sequence $\{u_p\}$ is bounded. Let $u^*$ be a weak-cluster point of $\{u_p\}.$
%and $j\in\mathbb{N}\cup\{0\}$ such that
%\begin{equation}\label{Eqeq3}
%  w-\lim_{j\to\infty}u_{p_j}=u^*.
%\end{equation}
Demiclosedness of $F_a-I$ and $G_b-I$ at zero, (\ref{Eqeq2}) and the weak convergence of a subsequence $\{u_{p_j}\}$ of $\{u_{p}\}$ to $u^*$ gives
\[
F_a(u^*)=u^*\quad \text{and}\quad G_b(Au^*)=Au^*.
\]
Thus, $u^*\in \{u\in Fix(F_a): Au\in Fix(G_b)\}$ and $Au^*\in Fix(G_b).$ Lemma \ref{Opial} implies that $\{u_p\}$
converges weakly to the split common fixed point $u^*.$

By Lemma \ref{L2a}, $Fix(F_a)=Fix(F),$  $Fix(G_b)=Fix(G)$ and
\[
(1-t_p)y_p+t_p F_{a}y_p=(1-a t_p)y_p+a t_p Fy_p.
\]
Since $G_b-I=b(G-I),$ putting $\alpha_p=a t_p,$ shows that $\{u_p\}$ generated by $(\ref{Alg1})$ converges weakly to a split common fixed point $u^*\in\Gamma.$
Note that since $0<a<1-\beta$ and $\{t_n\}\subset(\delta,1-\delta),$ it follows that $\{\alpha_p\}\subset (\delta,1-\delta).$
\end{proof}

\medskip

By considering a projection version of Algorithm \eqref{Alg1}, it is possible to obtain a strong convergence result as shown by the following theorem.

\begin{theorem} \label{MT2}
Let ${\cal H}_1$ and ${\cal H}_2$ be two real Hilbert spaces, $D$ a nonempty closed convex subset of ${\cal H}_1$ and $E$ a nonempty closed convex subset of ${\cal H}_2.$  Given a bounded linear operator $A:{\cal H}_1\to{\cal H}_2$, let $F: D\to {\cal H}_1,~ G : E\to {\cal H}_2$ be two $\beta$-demicontractive mappings
with $Fix(F)\neq\emptyset$ and $Fix(G)\neq\emptyset$ and both $F-I$ and $G-I$ are demiclosed at $0$. Let $\{u_n\}$ be a sequence generated in the following manner:
\begin{equation}\label{Alg2}
\begin{cases}
u_0\in D,\quad D_0=D, &\\
  z_p=P_D(u_p+\lambda A^*(b(G-I))Au_p),\quad b\in(0,1-\beta) &\\
  y_p=\alpha_p z_p+(1-\alpha_p)Fz_p,\quad \{\alpha_p\}\subset(0,\eta)\subset(0,1) &\\
  D_{p+1}=\{u\in D_p:\|y_p-u\|\leq \|z_p-u\|\leq\|u_p-u\|\}, & \\
  u_{p+1}=P_{D_{p+1}}(u_0), \qquad \forall p\in\mathbb{N}\cup \{0\},
\end{cases}
\end{equation}
where $P$ is a projection operator and $A^*$ denotes the adjoint of $A,$ $\lambda\in(0,1/\|A^*\|).$ If $\Gamma=\{u\in Fix(F): Au\in Fix(G)\}\neq \emptyset,$ then
\[
u_p\to u^*\in\Gamma ~\text{and}~ Au_p\to Au^*\in Fix(G).
\]
\end{theorem}

\begin{proof}
Since $F$ and $G$ are $\beta$-demicontractive mappings, it follows from Lemma \ref{L2} that the averaged mappings
\begin{align*}
F_{a}u=(1-a)u+a Fu \quad \text{and} \quad G_{b}u=(1-b)u+b Gu
\end{align*}
are quasi-nonexpansive for $\{a,b\}\subset (0,1-\beta).$ Clearly, $F_{a}-I$ and $G_{b}-I$ are also demiclosed at zero.

Now consider the sequence $\{u_p\}$ generated by \eqref{Alg2} %the following iteration:
%\begin{equation}\label{MTEq1}
%\begin{cases}
%u_0\in D,\quad D_0=D, &\\
 % z_p=P_D(u_p+\lambda A^*(G_{\nu}-I)Au_p), &\\
 % y_p=t_p z_p+(1-t_p)F_{\nu}z_p, \quad \{t_p\}\subset(0,\eta)\subset(0,1)&\\
  %D_{p+1}=\{u\in D_p:\|y_p-u\|\leq \|z_p-u\|\leq\|u_p-u\|\}, & \\
 % u_{p+1}=P_{D_{p+1}}(u_0),\qquad \forall p\in\mathbb{N}\cup \{0\},
%\end{cases}
%\end{equation}
It is straightforward to see that $D_p$ is closed and convex for any $p\in\mathbb{N}\cup \{0\}.$ Let $v\in\hat{\Gamma}=\{u\in Fix(F_a): Au\in Fix(G_b)\}.$ We have
\begin{align*}
  2\lambda & \langle u_p-v,A^*(G_{b}Au_p-Au_p)\rangle  \\
  = & ~2\lambda\langle A(u_p-v)+(G_{b}Au_p-Au_p)-(G_{b}Au_p-Au_p),G_{b}Au_p-Au_p\rangle \\
  = & ~2\lambda (\langle G_{b}Au_p-Av, G_{b}Au_p-Au_p\rangle-\|G_{b}Au_p-Au_p\|^2 \\
  = & ~2\lambda \left(\frac{1}{2}\|G_{b}Au_p-Av\|^2+\frac{1}{2}\|G_{b}Au_p-Au_p\|^2\right.\\
  &\qquad\left.-\frac{1}{2}\|Au_p-Av\|^2-\|G_{b}Au_p-Au_p\|^2\right) \\
  \leq & ~2\lambda \left(\frac{1}{2}\|G_{b}Au_p-Au_p\|^2-\|G_{b}Au_p-Au_p\|^2\|\right) \\
  = & -\lambda\|G_{b}Au_p-Au_p\|^2.
\end{align*}
Moreover,
\begin{align} \nonumber
  \|z_p-v\|^2&=\|P_D(u_p+\lambda A^*(G_{b}Au_p-P_D(v)\|^2 \\ \nonumber
  & \leq \|u_p+\lambda A^*(G_{b}Au_p-Au_p)-v\|^2 \\ \nonumber
   =& \|u_p-v\|^2+\|\lambda A^*(G_{b}Au_p-Au_p\|^2+2\lambda\langle u_p-v,A^*(G_{b}Au_p-Au_p)\rangle\\ \nonumber
  & \leq \|u_p-v\|^2+\lambda^2 \|A^*\|^2\|(G_{b}Au_p-Au_p\|^2-\lambda\|G_{b}Au_p-Au_p\|\\
  & = \|u_p-v\|^2+\lambda(1-\lambda\|A^*\|^2)\|(G_{b}Au_p-Au_p\|^2\| \label{Eqeq3}
\end{align}
Hence we have
\[
\|y_p-v\|\leq\|z_p-v\|\leq\|u_p-v\|
\]
This implies that $v\in D_p$ and $\hat{\Gamma}\subset D_p$ for each $p\in\mathbb{N}\cup\{0\}.$ For all $p\in\mathbb{N}$ we have
\[
\|u_{p+1}-u_0\|\leq \|v-u_0\|.
\]
Thus, $\{u_p\}$ is bounded and $\{\|u_p-u_0\|\}$ is nondecreasing. Therefore, $\{u_p\}$ is Cauchy. Let $u*=\lim_{p\to\infty}u_p.$ We obtain
\begin{align*}
  \|z_p-u_p\| &\leq 2\|u_{p+1}-u_p\|\to 0, \\
  \|y_p-u_p\| &\leq 2\|u_{p+1}-u_p\|\to 0, \\
  \|y_p-z_p\| & \leq \|y_p-u_p\|+\|u_p-z_p\|\to 0~ \text{as}~p\to\infty.
\end{align*}
Since $\lambda(1-\lambda\|A^*\|^2)>0,$ from (\ref{Eqeq3}) we obtain
\[
\|G_b Au_p-Au_p\|^2\leq \frac{1}{\lambda(1-\lambda\|A^*\|^2)}\|u_p-z_p\|\{\|u_p-v\|+\|z_n-v\|\}\to 0,
\]
which yields $Au^*\in Fix(G_b).$ Hence, $u^*\in\hat{\Gamma}$ and $\{u_p\}$ generated by (\ref{Alg2}) converges strongly to $u^*.$
By Lemma \ref{L2a}, for any $a,b\in(0,1),$ we have $Fix(F_{a})=Fix(F)$ and $Fix(G_{b})=Fix(G).$ Let $k_p=1-t_p,$ we have
\[
t_p z_p+(1-t_p)F_{a}z_p=(1-k_p) z_p+k_pF_{a}z_p=(1-a k_p) z_p+a k_p Fz_p.
\]
Since $G_b-I=b(G-I),$ putting $\alpha_p=1-a k_p$ shows that for the sequence $\{u_p\}$ generated by (\ref{Alg2}), we have
\[
u_p\to u^*\in\Gamma ~\text{and}~ Au_p\to Au^*\in Fix(G).
\]
Note that since $0<a,b<1-\beta$ and $\{t_p\}\subset(0,\eta)\subset(0,1),$ it follows that $$\{\alpha_p\}\subset(0,\eta)\subset(0,1).$$
\end{proof}

\medskip

In the following theorem, for $i=1,\cdots,n$ and $j=1,\cdots,k$ we denote
\[
\hat{F}_{i}=(1-a_i)I+a_i F_i \text{ and } \hat{G}_{j}=(1-b_j)I+b_j G_j, \quad a_i,b_j\in(0,1).
\]

\begin{theorem}
Let ${\cal H}_1$ and ${\cal H}_2$ be two real Hilbert spaces, $D$ a nonempty closed convex subset of ${\cal H}_1$ and $E$ a nonempty closed convex subset of ${\cal H}_2.$  Given a bounded linear operator $A:{\cal H}_1\to{\cal H}_2$, let $F_1,\cdots,F_n: D\to {\cal H}_1$ be $\beta$-demicontractive mappings with $\bigcap_{i=1}^n Fix(F_i)\neq\emptyset,$ and let $G_1,\cdots,G_k : {\cal H}_2\to {\cal H}_2$ be $\beta$-demicontractive mappings
with $\bigcap_{i=1}^k Fix(G_i)\neq\emptyset.$ Suppose $$(F_i-I),~~(i=1,\cdots,n) \text{ and } (G_i-I), ~~(i=1,\cdots,k)$$ are demiclosed at zero. Let $\{u_p\}$ be a sequence generated in the following manner:
\begin{equation}\label{Alg2a}
\begin{cases}
u_0\in D,\quad D_0=D, &\\
  z_p=P_D\left(u_p+\lambda A^*\left(\sum_{j=1}^{k}d_j{G}_{\varphi_j}-I\right)Au_p\right),\quad d_j,\varphi_j\in(0,1), \quad \sum_{j=1}^{k}d_j=1,&\\
  y_p=\alpha_p z_p+(1-\alpha_p)\sum_{i=1}^{n}c_i{F}_{\vartheta_i} z_p,\quad c_i,\vartheta_i\in(0,1),\quad \sum_{i=1}^{n}c_i=1,&\\
  D_{p+1}=\left\{u\in D_p:\|y_p-u\|\leq \|z_p-u\|\leq\|u_p-u\|\right\}, & \\
  u_{p+1}=P_{D_{p+1}}(u_0), \qquad \forall p\in\mathbb{N}\cup \{0\},
\end{cases}
\end{equation}
where $P$ is a projection operator and $A^*$ denotes the adjoint of $A,$ $$\lambda\in\left(0,\frac{1}{\|A^*\|^2}\right) \text{ and } \{\alpha_p\}\subset(0,\eta)\subset(0,1).$$ For $i=1,\cdots,n,$  $F_{\vartheta_i}:=(1-\vartheta_i)I+\vartheta_i\hat{F}_i$ and for $j=1,\cdots, k,$ let $G_{\varphi_j}=(1-\varphi_j)I+\varphi_j\hat{G}_j,$  and
$$\Omega=\left\{u\in\bigcap_{i=1}^{n}Fix(F_i): Au\in\bigcap_{j=1}^{k}Fix(G_j)\right\}.$$ Then the sequence $\{u_p\}$ converges strongly to $u^*\in\Omega.$
\label{MT3}
\end{theorem}
\begin{proof}
  Let $T_1=\sum_{i=1}^{n}c_i{F}_{\vartheta_i}, T_2=\sum_{j=1}^{k}d_jG_{\varphi_j}.$ 
  %Since $\hat{F}_i, i=1,\cdots,n,$ and $\hat{G_j}, j=1,\cdots,k,$ are all quasi-nonexpansive
Lemma \ref{LF} implies that $T_1$ and $T_2$ are quasi-nonexpansive. Furthermore, $$Fix(T_1)=\bigcap_{i=1}^{n}Fix(F_i)\neq\emptyset \text{ and } Fix(T_2)=\bigcap_{j=1}^{k}Fix(G_j)\neq\emptyset.$$ It is straightforward to verify that $T_1-I$ and $T_2-I$ are demiclosed at the origin. The rest of the proof is the same as the proof of Theorem \ref{MT2}.
\end{proof}

\medskip

As a direct consequence of Theorem \ref{MT2} we obtain strong convergence of the iteration to a common fixed point of $F$ and $G,$ considering ${\cal H}_1={\cal H}_2={\cal H}$ and $A$ as an identity mapping.
\begin{corollary}
Let ${\cal H}$ be a real Hilbert space, $D$ a nonempty closed convex subset of ${\cal H}.$ Let $F:D\to{\cal H}$ and $G:{\cal H}\to{\cal H}$ be two $\beta$-demicontractive mappings with $Fix(F)\cap Fix(G)\neq\emptyset.$ Assume that $F-I$ and $G-I$ are demiclosed at zero. Let $\{u_p\}$ be a sequence defined as follows:
\begin{equation*}
\begin{cases}
u_0\in D,\quad D_0=D, &\\
  z_p=P_D((1-\lambda)u_p+\lambda ((1-a)I+G)u_p),\quad a\in(0,1-\beta),~\lambda\in(0,1),&\\
  y_p=\alpha_p z_p+(1-\alpha_p)F z_p, \quad \{\alpha_p\}\subset(0,\eta)\subset(0,1),&\\
  D_{p+1}=\{u\in D_p:\|y_p-u\|\leq \|z_p-u\|\leq\|u_p-u\|\}, & \\
  u_{p+1}=P_{D_{p+1}}(u_0),\qquad \forall p\in\mathbb{N}\cup \{0\},
\end{cases}
\end{equation*}
where $P$ is a projection operator. Then $u_p\to u^*\in  Fix(F)\cap Fix(G).$
\end{corollary}
The next theorem is a generalization of Theorem \ref{MT2} in which the operators $F$ and $G$ are demicontractive with different constants. Its proof is similar to that of Theorem \ref{MT2}.
\begin{theorem} \label{MT4}
Let ${\cal H}_1$ and ${\cal H}_2$ be two real Hilbert spaces, $D$ a nonempty closed convex subset of ${\cal H}_1$ and $E$ a nonempty closed convex subset of ${\cal H}_2.$  Given a bounded linear operator $A:{\cal H}_1\to{\cal H}_2$, let $F: D\to {\cal H}_1$ be an $\alpha$-demicontractive mapping and $G : E\to {\cal H}_2$ be a $\beta$-demicontractive mapping %with $Fix(F)\neq\emptyset$ and $Fix(G) \neq\emptyset$ 
and both $F-I$ and $G-I$ are demiclosed at $0$. Let $\{u_n\}$ be a sequence generated in the following manner:
\begin{equation}\label{Alg4}
\begin{cases}
u_0\in D,\quad D_0=D, &\\
  z_p=P_D(u_p+\lambda A^*(b(G-I))Au_p),\quad b\in(0,1-\beta) &\\
  y_p=\alpha_p z_p+(1-\alpha_p)Fz_p,\quad \{\alpha_p\}\subset(0,\eta)\subset(0,1) &\\
  D_{p+1}=\{u\in D_p:\|y_p-u\|\leq \|z_p-u\|\leq\|u_p-u\|\}, & \\
  u_{p+1}=P_{D_{p+1}}(u_0), \qquad \forall p\in\mathbb{N}\cup \{0\},
\end{cases}
\end{equation}
where $P$ is a projection operator and $A^*$ denotes the adjoint of $A,$ $\lambda\in(0,1/\|A^*\|).$ If $\Gamma=\{u\in Fix(F): Au\in Fix(G)\}\neq \emptyset,$ then
\[
u_p\to u^*\in\Gamma ~\text{and}~ Au_p\to Au^*\in Fix(G).
\]
\end{theorem}

\section{Conclusions}
%= a brief description of the main results =
1. We have proven a weak convergence theorem for an iteration scheme used to approximate split common fixed point of two demicontractive mappings in Hilbert spaces derived from an associated weak convergence theorem in the class of quasi-nonexpansive operators. 

2. We also have shown a strong convergence theorem for an iteration scheme used to approximate split common fixed point of two demicontractive mappings in Hilbert spaces derived from a corresponding strong convergence theorem in the class of quasi-nonexpansive operators. 

3. Our investigation is based on an embedding technique by means of an average mappings: if $F$ is $\beta$-demicontractive, then for any $a\in(0,1-\beta),$ $F_a=(1-a)I+a F$ is quasi-nonexpansive.

4. For other related works that allow similar developments, we refer to Kingkam and Nantadilok \cite{Kingkam}, Padcharoen et al. \cite{Pad}, Sharma and Chandok \cite{Sharma}, Shi et al. \cite{Shi}, Tiammee and Tiamme \cite{Tiam},...

\section*{Acknowledgements} The first draft of this paper was carried out during the first author's short visit (December 2023) at the Department of Mathematics, King Fahd University of Petroleum and Minerals, Dhahran, Saudi Arabia. He is grateful to Professor Monther Alfuraidan, the Chairman of Department of Mathematics, for the invitation and for providing excellent facilities during his visit.

\end{document}